\documentclass[letterpaper, 11pt, reqno]{amsart}

\usepackage{amsmath,amssymb,amscd,amsthm,amsxtra,esint}
\usepackage{mathtools}
\usepackage{tikz}
\usepackage{hyperref}
\usepackage{cite}

\allowdisplaybreaks[2]
\newtheorem{theorem}{Theorem}[section]
\newtheorem{maintheorem}{Theorem}
\newtheorem{lemma}[theorem]{Lemma}
\newtheorem{proposition}[theorem]{Proposition}
\newtheorem{remark}[theorem]{Remark}

\newtheorem*{ackno}{Acknowledgments}

\newcommand{\N}{\mathbb N}
\newcommand{\C}{\mathbb C}
\newcommand{\T}{\mathbb T}
\newcommand{\R}{\mathbb R}
\newcommand{\Z}{\mathbb Z}
\newcommand{\la}{\langle}
\newcommand{\ra}{\rangle}
\newcommand{\norm}[1]{\|#1\|}
\newcommand{\abs}[1]{|#1|}
\newcommand{\FL}{\mathcal F L}
\newcommand{\Fx}{\mathcal F_x}

\numberwithin{equation}{section}
\numberwithin{theorem}{section}

\title[Norm inflation for the cubic HNLS on $\T^2$]
{Norm inflation for the cubic hyperbolic NLS on $\T^2$}

\author{Shunlin Shen}
\address{Shunlin Shen,
School of Mathematical Sciences, University of Science and Technology of China\\
Hefei, 230026\\
China}
\email{slshen@ustc.edu.cn}

\author{Yuzhao Wang}
\address{Yuzhao Wang,
School of Mathematical Sciences, Dalian University of Technology\\
Dalian, 116024\\
China;\\
School of Mathematics, University of Birmingham\\
Watson Building, Edgbaston\\
Birmingham B15 2TT\\
UK}
\email{wangyuzhao2008@gmail.com}

\subjclass[2020]{35Q55, 35B30}

\keywords{hyperbolic nonlinear Schr\"odinger equation; norm inflation;
ill-posedness; periodic solutions; sharp range}

\date{}

\begin{document}

\begin{abstract}
We prove norm inflation for the cubic hyperbolic nonlinear
Schr\"odinger equation in $H^s(\mathbb T^2)$
for every $s\in(-\infty,0)\cup(0,\frac12]$. The scaling-critical point $s=0$ is excluded by conservation of the
$L^2$ norm. The strong ill-posedness below and above the scaling-critical point arises from two completely different mechanisms. Particularly in the scaling-subcritical regime, this dynamical instability stems from the hyperbolic nature.
Together with the local well-posedness result in \cite{WangHNLS}, this
gives a sharp dichotomy away from the mass space $L^2(\mathbb T^2)$:
local well-posedness holds for $s>\frac12$, whereas norm inflation
occurs for all $s\le \frac12$ with $s\ne0$.
\end{abstract}

\maketitle

\section{Introduction}

We consider the cubic hyperbolic nonlinear Schr\"odinger equation (HNLS) on the
two-dimensional torus
\begin{equation}\label{eq:HNLS}
\begin{cases}
  i\partial_t u+(\partial_x^2-\partial_y^2)u+\abs{u}^2u=0, \\
 u|_{t = 0} = u_0,
  \end{cases}
   \qquad (t,x,y)\in\R\times\T^2,
\end{equation}
where $\T=\R/2\pi\Z$.

Hyperbolic Schr\"{o}dinger operators arise naturally in a variety of physical settings where the underlying geometry dictates an indefinite signature.
In
modulation theory, NLS-type equations describe the slowly varying envelope of
an oscillatory wave packet, and the second-order part of the envelope equation
is governed by the Hessian of the dispersion relation of the underlying
physical system.  When this Hessian is indefinite, the resulting envelope
equation is hyperbolic rather than elliptic. This situation appears for
deep-water gravity waves, for self-focusing wave packets in plasma and
nonlinear optics, and in related modulation models
\cite{AS79,B98,BK96,BR96,TotzHNLS}. It also connects HNLS with nonlocal
NLS-type systems such as the Davey--Stewartson and Ishimori systems, where
mixed elliptic-hyperbolic signatures enter either through the Schr\"odinger
operator or through the coupled mean-field equation
\cite{GhidagliaSautDS,HayashiSautDSIshimori,SW24}. In the water-wave setting,
the modulation approximation leading to HNLS has been rigorously justified
in important two- and three-dimensional regimes
\cite{TotzWuWaterWaves,TotzWaterWaves}. Additionally,
hyperbolic Schr\"{o}dinger operators also play a crucial role in kinetic theory. In fact, upon Fourier transformation in the velocity variable, the kinetic transport operator can be viewed as a hyperbolic Schr\"{o}dinger operator, as evidenced by recent dispersive analyses of the Boltzmann equation \cite{BSTW26,CDP19,CH24,CSZ24,CSZ26}.

Despite possessing one of the simplest structures among nonlinear dispersive equations, the cubic HNLS on
$\mathbb{T}^{2}$ displays analytical properties entirely distinct from those of the elliptic NLS. The main goal of this paper is to investigate how the non-elliptic structure and its associated dispersive effects intrinsically induce pathological behaviors of the solutions.

A natural starting point for understanding the analytic behavior of \eqref{eq:HNLS} is its scaling invariance. As with the elliptic NLS, the corresponding equation on $\R^2$ enjoys the scaling symmetry
\begin{align}\label{equ:nls, scaling}
u_{\lambda}(t,x)=\lambda u(\lambda^{2}t,\lambda x),\quad \lambda>0,
\end{align}
which yields the scaling-critical regularity $s_{c}=0$.
Based on standard scaling heuristics, initial data with regularity below this threshold should lead to ill-posedness, whereas regularity at or above this threshold suggests well-posedness. On the whole space $\mathbb{R}^{2}$, this heuristic indeed holds true.
Since the dispersive estimates for the hyperbolic Schr\"{o}dinger operator coincide with those for the elliptic case, the equation \eqref{eq:HNLS} is locally well-posed for $s\geq 0$; see \cite{TotzHNLS}.

On the torus, however, the situation is totally different. The mixed signature of the operator symbol gives rise to dispersive phenomena that are distinct from those in the elliptic NLS, as seen in \cite{BasakogluWangHNLS,BD15,BD17,DFZ25,GS1,GS2,GT12,MT15,SW24}. This difference is clearly reflected in the well-posedness regularity. For the cubic HNLS on $\mathbb{T}^{2}$, local well-posedness is only known for $s>\frac{1}{2}$ in \cite{WangHNLS}, far away from the scaling-critical regularity $s_{c}=0$. For $s \le \frac12$, it is also shown in \cite{WangHNLS,LiuZhengHNLS} that the solution map is not $C^3$. Failure of $C^3$ is ill-posedness in a weak sense, but it does not imply discontinuity of the solution map. 

This sharp threshold for $C^3$ regularity of the solution map, occurring strictly above the scaling-critical threshold, naturally raises a fundamental problem regarding whether the non-elliptic structure genuinely induces a new critical threshold at $s=\frac{1}{2}$ and whether the solutions exhibit pathological behavior below this threshold.
This paper provides an affirmative answer by establishing norm inflation,
  which serves as a mechanism of strong ill-posedness.

Let us recall the terminology of norm inflation in $H^s(\T^2)$
at the origin. This phenomenon occurs if, for every $\varepsilon>0$, there exist smooth initial data
$u_0$ and a time $t_\varepsilon\in(0,\varepsilon)$ such that
\[
  \norm{u_0}_{H^s(\T^2)}<\varepsilon,
  \qquad
  \norm{u(t_\varepsilon)}_{H^s(\T^2)}>\varepsilon^{-1},
\]
where $u$ is the corresponding smooth solution.  This is a much stronger statement than the failure of continuity of the solution map at the origin, capturing a severe dynamical instability where arbitrarily small initial data yield arbitrarily large norms in arbitrarily short times.

%

Our main result provides a complete characterization of this pathological behavior.
\begin{maintheorem}[Norm inflation]\label{thm:main}
The Cauchy problem \eqref{eq:HNLS} exhibits norm inflation at
the origin in $H^s(\T^2)$ if and only if
$s\in(-\infty,0)\cup(0,\frac12]$.
\end{maintheorem}

\begin{remark}\normalfont
\label{RMK:sharpness}
The second author~\cite{WangHNLS} proved analytic local well-posedness of the
cubic HNLS on $\T^2$ in $H^s(\T^2)$ for $s>\frac12$, thereby effectively precluding norm inflation in this regime. On the other hand,
\cite{WangHNLS} also proved ill-posedness for $s<\frac{1}{2}$ in the weaker sense that the solution map fails
to be $C^3$ at the origin.  Liu and Zheng~\cite{LiuZhengHNLS} proved an
analogous failure of $C^3$ regularity for the cubic HNLS on $\T^2$ at the
endpoint $H^{1/2}$. 
\end{remark}

\begin{remark}\rm
The mechanism in Theorem~\ref{thm:main} is robust and extends beyond the two-dimensional cubic HNLS considered here. 
Indeed, our two-dimensional construction can be embedded into the three-dimensional HNLS. 
Together with the local well-posedness result for $s>\frac12$ in \cite{LiuZhengHNLS}, this
gives a complete characterization of the norm-inflation range for the HNLS on $\mathbb T^3$, namely $s\in(-\infty,0)\cup(0,\frac12]$.

The same null-direction mechanism also applies to several related hyperbolic models. In particular, it yields norm inflation for Davey--Stewartson type systems on $\mathbb T^2$ \cite{LZ26,Godet13}, as well as for the quintic HNLS on $\mathbb T^2$ and the cubic HNLS on $\mathbb T^4$ with hyperbolic Laplacian
$\partial_{x_1}^2+\partial_{x_2}^2-\partial_{x_3}^2-\partial_{x_4}^2$,
both considered in \cite{LZ26,BasakogluWangHNLS}. 
Together with the corresponding well-posedness theories established in
\cite{LZ26,Godet13,BasakogluWangHNLS}, these results identify the sharp
norm-inflation ranges for the above models.
In the four-dimensional case, restricting to the null plane $(x_1,x_2,x_3,x_4) \mapsto (x_1+x_3,x_2+x_4)$
reduces the equation exactly to the pointwise cubic ODE, thereby yielding norm inflation at the critical regularity $H^1(\T^4)$.
\end{remark}

The novelty of Theorem~\ref{thm:main} lies in deepening the characterization of ill-posedness for the hyperbolic Schr\"{o}dinger equation from topological discontinuity of the solution map to the dynamical ill-posed behavior of the solutions themselves. Previous works \cite{WangHNLS, LiuZhengHNLS} proved that the solution map is not $C^{3}$ at the origin, which is essentially an irregularity of the data-to-solution correspondence at the level of functional analysis and does not necessarily imply uncontrolled behavior of the physical solutions. Norm inflation, by contrast, directly points to an intrinsic dynamical instability.
Even for arbitrarily small initial perturbations,
the interplay between nonlinearity and hyperbolic dispersion can drive the Sobolev norm to arbitrarily large values within an arbitrarily short time. Consequently, Theorem~\ref{thm:main} confirms that throughout the interval $s\leq \frac{1}{2}$, except for the mass-conserved endpoint $s=0$, the equation \eqref{eq:HNLS} not only possesses a highly rough solution map but also admits solutions with uncontrollable pathological behavior.

A particularly notable feature of Theorem~\ref{thm:main} is the discontinuous nature of the ill-posedness phase diagram arising from the hyperbolic structure.
 Based on scaling heuristics, the scaling-critical index $s=0$ should serve as the natural threshold separating well-posedness from ill-posedness.
 Even in some elliptic cases, the actual threshold may lie strictly above the scaling-critical value, as seen in the work \cite{KPV01} of Kenig, Ponce, and Vega. However, the current hyperbolic setting falls outside both expectations. Here,
 $s=0$ acts as an isolated point strictly protected by mass conservation and remains entirely free from norm inflation. Hence the ill-posedness phase diagram becomes unusual, with norm inflation absent exactly at the scaling-critical regularity and strong ill-posedness appearing on both sides. This feature, driven by the hyperbolic structure, stands in sharp contrast to the elliptic NLS.

\vspace{1em}

\noindent\textbf{Outline of the proof}.
The study of norm inflation was developed systematically by Christ,
Colliander, and Tao
\cite{CCT2b}, who proved norm inflation for nonlinear Schr\"odinger and wave
equations below the scaling critical regularity. 
A closely related approach proves ill-posedness through
high-to-low transfer in the first Picard iterate; see Bejenaru--Tao
\cite{BT}, Iwabuchi--Ogawa \cite{IO}, Kishimoto \cite{Kishimoto}, and Oh--Wang \cite{OhWang2015}. 
Further
norm-inflation phenomena for periodic NLS in Fourier-Lebesgue spaces were
obtained in \cite{CK}. 
In a different geometric direction, Burq,
G\'erard, and Tzvetkov \cite[Appendix]{BGT2} used dispersionless dynamics to
prove an instability in $H^1$ for super-quintic NLS on three-dimensional
compact manifolds.

The negative
regularity part $s<0$ in
Theorem~\ref{thm:main}, which falls below the scaling-critical index, fits the familiar
scaling-supercritical philosophy behind norm-inflation results discussed above. In that
range, the proof uses a standard high-to-low transfer mechanism: high-frequency
waves generate a large low-frequency component on a short time scale. See
Section~\ref{sec:negative} for further details.

The positive regularity part $s\in(0,\frac12]$ of
Theorem~\ref{thm:main} is of a different nature.
This range is scaling-subcritical, so scaling heuristics alone do not predict
norm inflation. The mechanism is not the usual scaling-supercritical
high-frequency cascade. Rather, it comes from the hyperbolic nature of
\eqref{eq:HNLS}. In fact, the symbol $\xi_1^2-\xi_2^2$ has a null direction, along which the equation reduces exactly to the pointwise ODE.  This
null-direction mechanism gives norm inflation for $0<s\le\frac12$.  More
precisely, we use exact solutions concentrated on the null direction $x-y$.
For $0<s<\frac12$, the data are normalized Dirichlet kernels
\[
  D_N(z):=\sum_{\abs k\le N}e^{ikz}
  =\frac{\sin((N+\frac12)z)}{\sin(z/2)}
\]
and the large parameter is the power gain $N^{1-2s}$.  At the endpoint
$s=\frac12$, this power gain disappears; we instead use the logarithmic profile
\[
  F_N(z)=\sum_{k=1}^N\frac{e^{ikz}}k,
\]
which has bounded $\dot H^{1/2}$ size after normalization but retains a
logarithmic pointwise height.  The endpoint lower bound works directly with the
full exponential phase and accumulates over logarithmic spatial scales.
See Sections~\ref{sec:subcritical} and~\ref{sec:critical} for further details.

\vspace{1em}

\noindent\textbf{Organization of the paper}.
In Section~\ref{sec:prep}, we collect some tools that will be used throughout the paper.  Section~\ref{sec:negative}
proves norm inflation in the negative regularity range $s<0$ by comparing the
HNLS interaction representation with the pointwise ODE flow.  In
Section~\ref{sec:subcritical}, we prove the positive subcritical case
$0<s<\frac12$ using normalized Dirichlet kernels along a null direction.
Finally, Section~\ref{sec:critical} treats the endpoint $s=\frac12$ using a
logarithmic Fourier profile and the double-integral characterization of
$\dot H^{1/2}(\T)$.

\section{Preparation}\label{sec:prep}

We begin by fixing notation and recording the estimates used in the three
constructions below. For nonnegative quantities $X$ and $Y$, we write
$X\lesssim Y$ if $X\le CY$ for a constant $C>0$, and $X\sim Y$ if both
$X\lesssim Y$ and $Y\lesssim X$. A subscript indicates the permitted
dependence of the implicit constant; for example, $X\lesssim_s Y$. Unless
stated otherwise, all implicit constants are independent of the large
frequency parameters.

For $d=1,2$, we write the Fourier series on $\T^d$ as
\[
  f(x)=\sum_{n\in\Z^d}\widehat f(n)e^{in\cdot x},
  \qquad
  \widehat f(n)=\frac{1}{(2\pi)^d}\int_{\T^d}f(x)e^{-in\cdot x}\,dx.
\]
When the expression inside the Fourier transform is long, we also write
\[
  (\Fx f)(n):=\widehat f(n).
\]
For $s\in\R$, the inhomogeneous Sobolev norm is
\[
  \norm{f}_{H^s(\T^d)}^2
  :=\sum_{n\in\Z^d}\la n\ra^{2s}\abs{\widehat f(n)}^2,
  \qquad
  \la n\ra=(1+\abs{n}^2)^{1/2}.
\]
For $s>0$, we also write the homogeneous Sobolev norm
\[
  \norm{f}_{\dot H^s(\T^d)}^2
  :=\sum_{n\in\Z^d\setminus\{0\}}\abs{n}^{2s}\abs{\widehat f(n)}^2.
\]
For $s>0$, the inhomogeneous and homogeneous Sobolev norms satisfy
\[
\norm{f}_{H^s(\T^d)} \sim  \norm{f}_{L^2(\T^d)} +   \norm{f}_{\dot H^s(\T^d)}.
\]

\subsection{Fourier-Lebesgue estimates}

For $1\le p\le\infty$, define the Fourier-Lebesgue norm by
\[
  \norm{f}_{\FL^p(\T^d)}
  :=
  \norm{\widehat f}_{\ell^p(\Z^d)}.
\]
Thus, for $1\le p<\infty$,
\[
  \norm{f}_{\FL^p(\T^d)}
  =
  \bigg(\sum_{n\in\Z^d}\abs{\widehat f(n)}^p\bigg)^{1/p},
\]
while
\[
  \norm{f}_{\FL^\infty(\T^d)}
  =
  \sup_{n\in\Z^d}\abs{\widehat f(n)}.
\]
In particular, the two endpoint cases used frequently below are
\[
  \norm{f}_{\FL^1(\T^d)}=\sum_{n\in\Z^d}\abs{\widehat f(n)},
  \qquad
  \norm{f}_{\FL^\infty(\T^d)}=\sup_{n\in\Z^d}\abs{\widehat f(n)}.
\]
We suppress the domain $\T^d$ from the notation when no confusion can arise.
The elementary convolution estimates
\begin{align}
  \norm{fg}_{\FL^1}&\le \norm{f}_{\FL^1}\norm{g}_{\FL^1},\label{eq:FL1-alg}\\
  \norm{fg}_{\FL^\infty}
  &\le \norm{f}_{\FL^\infty}\norm{g}_{\FL^1}\label{eq:FLinf-alg}
\end{align}
will be used repeatedly.
We also need their trilinear counterpart. Suppose
that
$\abs{m(n,n_1,n_2,n_3)}\le1$ and
\[
  \big(\Fx\mathcal N_m(f_1,f_2,f_3)\big)(n)
  :=
  \sum_{n=n_1-n_2+n_3}
  m(n,n_1,n_2,n_3)
  \widehat f_1(n_1)\overline{\widehat f_2(n_2)}\widehat f_3(n_3),
\]
where $\mathcal F_x$ denotes the spatial Fourier transform,
then
\begin{align}
  \norm{\mathcal N_m(f_1,f_2,f_3)}_{\FL^1}
  &\le
  \norm{f_1}_{\FL^1}\norm{f_2}_{\FL^1}\norm{f_3}_{\FL^1},
  \label{eq:tri-FL1}\\
  \norm{\mathcal N_m(f_1,f_2,f_3)}_{\FL^\infty}
  &\le
  \norm{f_1}_{\FL^\infty}\norm{f_2}_{\FL^1}\norm{f_3}_{\FL^1}.
  \label{eq:tri-FLinf}
\end{align}
The same $\FL^\infty$ bound holds with the $\FL^\infty$ norm placed on any
one of the three factors.  These estimates follow directly by taking absolute
values in the defining convolution sums.

\begin{lemma}[Short-time Wiener bounds]\label{lem:wiener-bounds}
Let $\phi\in \FL^1(\T^2)$ and set
\[
  M:=\norm{\phi}_{\FL^1(\T^2)},\qquad
  A:=\norm{\phi}_{\FL^\infty(\T^2)}.
\]
Suppose that $v$ is a smooth function whose Fourier coefficients satisfy,
for $0\le t\le T$,
\[
  \widehat v(n,t)=\widehat\phi(n)
  +i\int_0^t
  \sum_{n=n_1-n_2+n_3}
  m(\bar n,t')
  \widehat v(n_1,t')\overline{\widehat v(n_2,t')}
  \widehat v(n_3,t')\,dt',
\]
where $\bar n=(n,n_1,n_2,n_3)$ and $\abs{m(\bar n,t')}\le1$.  If
$TM^2\le c_0$ for a sufficiently small absolute constant $c_0>0$, then
\[
  \sup_{0\le t\le T}\norm{v(t)}_{\FL^1}\lesssim M,
  \qquad
  \sup_{0\le t\le T}\norm{v(t)}_{\FL^\infty}\lesssim A,
\]
and, for $0\le t\le T$,
\[
  \norm{v(t)-\phi}_{\FL^1}\lesssim tM^3,
  \qquad
  \norm{v(t)-\phi}_{\FL^\infty}\lesssim tAM^2.
\]
\end{lemma}

\begin{proof}
Let
\[
  X(t):=\sup_{0\le t'\le t}\norm{v(t')}_{\FL^1}.
\]
If $M=0$, then $\phi=0$, and \eqref{eq:tri-FL1} gives
$X(t)\le CtX(t)^3$.  A standard continuity argument then yields
$v\equiv0$ on $[0,T]$.  Hence we may assume that $M>0$.

By \eqref{eq:tri-FL1}, for every $0\le \tau\le t\le T$,
\[
  \norm{v(\tau)}_{\FL^1}
  \le M+C\int_0^\tau\norm{v(t')}_{\FL^1}^3\,dt'
  \le M+CtX(t)^3.
\]
We close this estimate by a bootstrap.  If $X(t)\le2M$, then, since
$tM^2\le TM^2\le c_0$,
\[
  X(t)\le M+8CtM^3\le M+8Cc_0M.
\]
Choosing $c_0$ so that $8Cc_0\le\frac12$ improves the bound to
$X(t)\le\frac32M$.  By continuity, the bootstrap remains valid throughout
$[0,T]$; in particular,
\[
  \sup_{0\le t\le T}\norm{v(t)}_{\FL^1}\lesssim M.
\]

Next set
\[
  Y(t):=\sup_{0\le t'\le t}\norm{v(t')}_{\FL^\infty}.
\]
Using \eqref{eq:tri-FLinf} and the bound just proved, we obtain
\[
  Y(t)\le A+CM^2\int_0^t Y(t')\,dt'.
\]
Gronwall's inequality gives
$Y(t)\le A e^{CM^2t}\le A e^{Cc_0}\lesssim A$ uniformly on $[0,T]$.

It remains to estimate the increments from the initial datum. From the
integral equation,
\[
  v(t)-\phi
  =
  i\int_0^t \mathcal N_{m(t')}(v(t'),v(t'),v(t'))\,dt',
\]
where $m(t')$ denotes the bounded multiplier at time $t'$.  Therefore,
\eqref{eq:tri-FL1} and the $\FL^1$ bound imply
\[
  \norm{v(t)-\phi}_{\FL^1}
  \le
  \int_0^t
  \norm{\mathcal N_{m(t')}(v(t'),v(t'),v(t'))}_{\FL^1}\,dt'
  \lesssim
  \int_0^t M^3\,dt'
  \lesssim tM^3.
\]
Similarly, \eqref{eq:tri-FLinf}, with the $\FL^\infty$ norm placed on one
copy of $v(t')$, and the bounds already proved give
\[
\begin{aligned}
  \norm{v(t)-\phi}_{\FL^\infty}
  &\le
  \int_0^t
  \norm{\mathcal N_{m(t')}(v(t'),v(t'),v(t'))}_{\FL^\infty}\,dt'
  \\
  &\lesssim
  \int_0^t AM^2\,dt'
  \lesssim tAM^2.
\end{aligned}
\]
This proves the lemma.
\end{proof}

\subsection{Null-direction solutions}

The positive-regularity constructions exploit exact one-dimensional
solutions supported along a null direction of the hyperbolic symbol. Let
$z=x-y$.
If $u(t,x,y)=f(t,z)$, then
\[
  (\partial_x^2-\partial_y^2)f(t,x-y)
  =\partial_z^2 f(t,z)-\partial_z^2 f(t,z)=0.
\]
Thus \eqref{eq:HNLS} reduces to the pointwise ODE
\begin{equation}
  i\partial_t f+\abs{f}^2f=0.
\end{equation}
For any smooth initial profile $f_0$, the solution is
\begin{equation}
  f(t,z)=f_0(z)e^{it\abs{f_0(z)}^2}.
\end{equation}
Consequently,
\[
  u(t,x,y)=f_0(x-y)e^{it\abs{f_0(x-y)}^2}
\]
is an exact solution to \eqref{eq:HNLS}.

\begin{lemma}[Reduction from $\T^2$ to $\T$]\label{lem:T2-to-T}
Let $s\in\R$.  If $u(x,y)=f(x-y)$, then
\[
  \norm{u}_{H^s(\T^2)}\sim_s \norm{f}_{H^s(\T)}.
\]
\end{lemma}

\begin{proof}
Write $f(z)=\sum_{k\in\Z}\widehat f(k)e^{ikz}$.  Then
\[
  f(x-y)=\sum_{k\in\Z}\widehat f(k)e^{ikx}e^{-iky},
\]
so $\widehat u(k,\ell)=0$ unless $\ell=-k$, and
$\widehat u(k,-k)=\widehat f(k)$.  Therefore
\[
  \norm{u}_{H^s(\T^2)}^2
  =\sum_{k\in\Z}\la(k,-k)\ra^{2s}\abs{\widehat f(k)}^2
  \sim_s \sum_{k\in\Z}\la k\ra^{2s}\abs{\widehat f(k)}^2
  =\norm{f}_{H^s(\T)}^2,
\]
where the implicit constant depends only on $s$.
\end{proof}

\subsection{Sobolev difference estimates}

We record two ways to bound a fractional Sobolev norm from below. The first
uses one fixed spatial increment and will be applied in the subcritical case $s < \frac12$.
The second averages over all increments and is needed at the endpoint $s = \frac12$.

\begin{lemma}[Difference estimate]\label{lem:fixed-increment}
Let $0<s<1$.  There exists $c_s>0$ such that, for every $f\in H^s(\T)$ and
every $h\in\T$ with $0<\abs h\le\pi$,
\[
  \norm{f}_{\dot H^s(\T)}^2
  \ge c_s\abs h^{-2s}\norm{f(\cdot+h)-f(\cdot)}_{L^2(\T)}^2.
\]
\end{lemma}

\begin{proof}
By Plancherel,
\[
  \norm{f(\cdot+h)-f(\cdot)}_{L^2(\T)}^2
  =2\pi\sum_{k\in\Z}\abs{e^{ikh}-1}^2\abs{\widehat f(k)}^2.
\]
Since $\abs{e^{ix}-1}\le C\min\{\abs x,1\}\le C_s\abs x^s$ for
$0<s<1$,
\[
  \norm{f(\cdot+h)-f(\cdot)}_{L^2(\T)}^2
  \le C_s\abs h^{2s}\sum_{k\in\Z}\abs{k}^{2s}\abs{\widehat f(k)}^2.
\]
Rearranging gives the claim.
\end{proof}

The characterization of homogeneous Sobolev norms by the
$L^2$-modulus of continuity is classical in the Euclidean setting; see
H\"ormander~\cite{Hormander}. Its periodic analogue in arbitrary dimension
is given by B\'enyi and Oh~\cite[Proposition~1.3]{BenyiOh}. We include the
short one-dimensional proof for completeness.

\begin{lemma}[Difference quotients]\label{lem:Hs-difference}
Let $0<s<1$.  For every smooth $h:\T\to\C$,
\[
  \norm{h}_{\dot H^s(\T)}^2
  \sim_s
  \int_{-\pi}^{\pi}
  \frac{\norm{h(\cdot+r)-h(\cdot)}_{L^2(\T)}^2}{\abs{r}^{1+2s}}\,dr
  \sim_s
  \iint_{\T\times\T}
  \frac{\abs{h(z)-h(w)}^2}{d_\T(z,w)^{1+2s}}\,dz\,dw,
\]
where
\[
d_\T(z,w):=\min_{\ell\in\Z}\abs{z-w+2\pi\ell}
\]
is the geodesic distance on $\T$.
\end{lemma}

\begin{proof}
Set
\[
  A_s(h):=\int_{-\pi}^{\pi}
  \frac{\norm{h(\cdot+r)-h(\cdot)}_{L^2(\T)}^2}{\abs{r}^{1+2s}}\,dr.
\]
Writing $h(z)=\sum_{k\in\Z}\widehat h(k)e^{ikz}$ and using Plancherel,
\[
  A_s(h)=2\pi
  \sum_{k\in\Z}\abs{\widehat h(k)}^2
  \int_{-\pi}^{\pi}\frac{\abs{e^{ikr}-1}^2}{\abs{r}^{1+2s}}\,dr.
\]
For $k=0$ the inner integral vanishes.  For $k\ne0$, the change of variables
$v=kr$ gives
\[
  \int_{-\pi}^{\pi}\frac{\abs{e^{ikr}-1}^2}{\abs{r}^{1+2s}}\,dr
  =\abs{k}^{2s}\int_{-\pi\abs k}^{\pi\abs k}
  \frac{\abs{e^{iv}-1}^2}{\abs{v}^{1+2s}}\,dv
  \sim_s \abs{k}^{2s},
\]
uniformly in $k\ne0$.  Indeed, the upper bound follows from
$\abs{e^{iv}-1}\lesssim\min\{\abs v,1\}$, while the lower bound follows by
integrating over the fixed interval $[1,2]\subset[0,\pi]$.  Hence
$A_s(h)\sim_s\norm{h}_{\dot H^s(\T)}^2$.

The double-integral form follows by writing $r=w-z$, with the geodesic
representative $r\in[-\pi,\pi]$.
\end{proof}

We also use the elementary fact that, for every $0<a\le b<2\pi$, there is
$c_{a,b}>0$ such that
\begin{equation}\label{eq:phase-away}
  \abs{e^{i\theta}-1}\ge c_{a,b}
  \qquad\text{whenever}\qquad
  a\le\abs\theta\le b.
\end{equation}
Indeed, this follows from
$\abs{e^{i\theta}-1}=2\abs{\sin(\theta/2)}$ and the compactness of
$[a,b]\subset(0,2\pi)$.

\section{The negative regularity case \texorpdfstring{$s<0$}{s<0}}
\label{sec:negative}

We first prove Theorem~\ref{thm:main} in the negative regularity range
$s<0$.  Unlike the positive-regularity argument below in Sections~\ref{sec:subcritical} and~\ref{sec:critical}, this part does not use
an exact null-direction solution.  Instead, following the ODE-approximation
strategy in~\cite[Section~3.3]{OhWang2015} by Oh and the second author, we choose data supported on two
high-frequency blocks.  The cubic interaction (see \eqref{QN} for the definition of $Q_N$),
\[
  Q_N-Q_{2N}+Q_N
\]
produces a large low-frequency component in the pointwise ODE flow.  On the
short time scale used below, the HNLS interaction representation remains close
to this ODE flow, since the oscillatory factor in the first Picard iterate
satisfies $e^{it\Phi}=1+o(1)$.

Throughout this section, fix $s<0$.  Let $N\gg1$ be dyadic and let
$1\ll R\ll N$ be another dyadic parameter to be chosen later.  Write
$e_1=(1,0)$ and define
\begin{align}\label{QN}
\begin{aligned}
  Q_N&:=\{n\in\Z^2:\abs{n-Ne_1}_\infty\le R/10\},\\
  Q_{2N}&:=\{n\in\Z^2:\abs{n-2Ne_1}_\infty\le R/10\}.
\end{aligned}
\end{align}
Let $\alpha_N>0$ and define the trigonometric polynomial $\phi_N$ by
\begin{equation}\label{eq:neg-data}
  \widehat{\phi_N}(n)
  :=\alpha_N\big(\mathbf{1}_{Q_N}(n)+\mathbf{1}_{Q_{2N}}(n)\big).
\end{equation}
\begin{lemma}[Size of the high-frequency data]\label{lem:neg-data-size}
For $\phi_N$ defined by \eqref{eq:neg-data},
\[
  \norm{\phi_N}_{H^s(\T^2)}\sim_s \alpha_N R N^s,
  \qquad
  \norm{\phi_N}_{\FL^1(\T^2)}\sim \alpha_N R^2,
  \qquad
  \norm{\phi_N}_{\FL^\infty(\T^2)}=\alpha_N.
\]
\end{lemma}

\begin{proof}
The support of $\widehat{\phi_N}$ is contained in two boxes of cardinality
comparable to $R^2$, and all these frequencies have size comparable to $N$.
Therefore,
  $\norm{\phi_N}_{H^s(\T^2)}^2
  \sim_s \alpha_N^2N^{2s}R^2$,
which gives the Sobolev estimate.  The two Fourier-Lebesgue estimates follow
directly from \eqref{eq:neg-data}.
\end{proof}

Let
\[
  Q_0:=\{n\in\Z^2:\abs n_\infty\le R/100\},
  \qquad
  \mathcal L_s(R):=
  \bigg(\sum_{n\in Q_0}\la n\ra^{2s}\bigg)^{1/2}.
\]
A direct computation gives
\begin{equation}\label{eq:neg-Ls}
  \mathcal L_s(R)\sim_s
  \begin{cases}
    R^{1+s},& -1<s<0,\\
    (\log R)^{1/2},& s=-1,\\
    1,& s<-1.
  \end{cases}
\end{equation}

\begin{lemma}
\label{lem:neg-cubic-count}
For every $n\in Q_0$, we have
\[
  \big(\Fx(\abs{\phi_N}^2\phi_N)\big)(n)
  =
  \sum_{n=n_1-n_2+n_3}
  \widehat{\phi_N}(n_1)\overline{\widehat{\phi_N}(n_2)}
  \widehat{\phi_N}(n_3)
  \gtrsim \alpha_N^3R^4.
\]
\end{lemma}

\begin{proof}
By \eqref{eq:neg-data},
\[
\begin{aligned}
&\sum_{n=n_1-n_2+n_3}
  \widehat{\phi_N}(n_1)\overline{\widehat{\phi_N}(n_2)}
  \widehat{\phi_N}(n_3) \\
&\qquad =
  \alpha_N^3
  \#\big\{(n_1,n_2,n_3)\in(Q_N\cup Q_{2N})^3:
  n=n_1-n_2+n_3\big\}.
\end{aligned}
\]
We count only triples with
\[
  n_1\in Q_N,\qquad n_2\in Q_{2N},\qquad n_3\in Q_N.
\]
For fixed $n\in Q_0$, choose $n_1,n_3$ in the smaller boxes
\[
  \abs{n_1-Ne_1}_\infty\le R/1000,
  \qquad
  \abs{n_3-Ne_1}_\infty\le R/1000.
\]
There are $\gtrsim R^4$ such choices.  For each pair, define
$n_2=n_1+n_3-n$.  Then
\[
  \abs{n_2-2Ne_1}_\infty
  \le
  \abs{n_1-Ne_1}_\infty
  +\abs{n_3-Ne_1}_\infty
  +\abs n_\infty
  <R/10,
\]
so $n_2\in Q_{2N}$.  Each of these triples contributes $\alpha_N^3$, and all
contributions have the same sign because the Fourier coefficients in
\eqref{eq:neg-data} are positive.
\end{proof}

Let $w_N$ be the solution to the pointwise ODE
\begin{align}\label{eq:ODE0}
  i\partial_t w+\abs{w}^2w=0,
  \qquad
  w(0)=\phi_N.
\end{align}
Thus
\begin{align}\label{eq:ODE}
  w_N(t)=\phi_Ne^{it\abs{\phi_N}^2}.
\end{align}
For a finite set $E\subset\Z^2$, let $P_E$ denote the Fourier projection onto
$E$.

\begin{lemma}[ODE low-frequency lower bound]\label{lem:neg-ODE-lower}
Let $\beta_N\in(0,1)$ with $\beta_N\to0$ as $N\to \infty$, and set
\[
  T_N:=\beta_N \norm{\phi_N}_{\FL^1(\T^2)}^{-2}.
\]
Then, for all sufficiently large $N$,
\[
  \norm{P_{Q_0}w_N(T_N)}_{H^s(\T^2)}
  \gtrsim_s \beta_N\alpha_N\mathcal L_s(R).
\]
\end{lemma}

\begin{proof}
Expanding the exact ODE solution \eqref{eq:ODE} by Taylor formula gives
\[
  w_N(T_N)=\phi_N+iT_N\abs{\phi_N}^2\phi_N+\mathcal R_N,
\]
where
\[
  \mathcal R_N=\phi_N\sum_{k\ge2}\frac{(iT_N\abs{\phi_N}^2)^k}{k!}.
\]
Since $\widehat{\phi_N}$ is supported in $Q_N\cup Q_{2N}$, the initial term $\phi_N$
has no Fourier support in $Q_0$.  By Lemmas~\ref{lem:neg-data-size} and \ref{lem:neg-cubic-count},
\[
  \abs{\big(\Fx(T_N\abs{\phi_N}^2\phi_N)\big)(n)}
  \gtrsim
  T_N\alpha_N^3R^4
  =
  \beta_N \norm{\phi_N}_{\FL^1}^{-2}\alpha_N^3R^4
  \sim
  \beta_N\alpha_N
\]
for every $n\in Q_0$.  On the other hand,
\eqref{eq:FL1-alg} and \eqref{eq:FLinf-alg} imply
\[
\begin{split}
\norm{\mathcal R_N}_{\FL^\infty}
  & \le
  \norm{\phi_N}_{\FL^\infty} \sum_{k\ge2}
  \frac{\big(T_N\norm{\phi_N}_{\FL^1}^2\big)^k}{k!} \\
  &  =
  \alpha_N\sum_{k\ge2}
  \frac{\big(T_N\norm{\phi_N}_{\FL^1}^2\big)^k}{k!}
  =
  \alpha_N\sum_{k\ge2}
  \frac{\beta_N^k}{k!}
  \lesssim \alpha_N\beta_N^2.
\end{split}
\]
For large $N$, this is at most half of the cubic contribution
$\beta_N\alpha_N$ on $Q_0$.
Therefore
\[
  \abs{\widehat{w_N(T_N)}(n)}\gtrsim \beta_N\alpha_N,
  \qquad n\in Q_0,
\]
and hence
\[
  \norm{P_{Q_0}w_N(T_N)}_{H^s(\T^2)}
  =
  \bigg(\sum_{n\in Q_0}\la n\ra^{2s}
  \abs{\widehat{w_N(T_N)}(n)}^2\bigg)^{1/2}
  \gtrsim_s \beta_N\alpha_N\mathcal L_s(R).
\]
This proves the lemma.
\end{proof}

We next compare the HNLS solution with the ODE flow.  Let $u_N$ denote the HNLS
solution with $u_N(0)=\phi_N$, and define the interaction representation
\[
  \mathbf u_N(t):=e^{-it(\partial_x^2-\partial_y^2)}u_N(t).
\]
In Fourier variables,
\begin{equation}\label{eq:neg-int-rep}
  \widehat{\mathbf u_N}(n,t)
  =
  \widehat{\phi_N}(n)
  +i\sum_{n=n_1-n_2+n_3}
  \int_0^t
  e^{it'\Phi(\bar n)}
  \widehat{\mathbf u_N}(n_1,t')
  \overline{\widehat{\mathbf u_N}(n_2,t')}
  \widehat{\mathbf u_N}(n_3,t')\,dt',
\end{equation}
where
\[
  q(k):=k_1^2-k_2^2,
  \qquad
  \Phi(\bar n):=q(n)-q(n_1)+q(n_2)-q(n_3),
  \qquad
  n=n_1-n_2+n_3.
\]
Note that the ODE solution to \eqref{eq:ODE0} satisfies the same integral equation with
$e^{it'\Phi(\bar n)}$ replaced by $1$,
\begin{align}\label{eq:neg-int-rep-2}
  \widehat{w_N}(n,t)
  =
  \widehat{\phi_N}(n)
  +i\sum_{n=n_1-n_2+n_3}
  \int_0^t
  \widehat{w_N}(n_1,t')
  \overline{\widehat{w_N}(n_2,t')}
  \widehat{w_N}(n_3,t')\,dt',
\end{align}
The next lemma shows that, when $t$ is small compared with the oscillation
scale of $\Phi$, i.e. $|t\Phi| \ll 1$, the ODE \eqref{eq:neg-int-rep-2} is a good approximation to
the HNLS interaction representation \eqref{eq:neg-int-rep}.

\begin{lemma}[ODE approximation]\label{lem:neg-ODE-approx}
Assume
\[
  T_N\norm{\phi_N}_{\FL^1(\T^2)}^2=\beta_N \to0,
  \qquad
  T_NN^2\to0.
\]
Then
\[
  \norm{\mathbf u_N-w_N}_{L^\infty([0,T_N];\FL^\infty)}
  =o(\beta_N\alpha_N),
\]
and consequently
\[
  \norm{P_{Q_0}(\mathbf u_N(T_N)-w_N(T_N))}_{H^s(\T^2)}
  =o(\beta_N\alpha_N) \mathcal L_s(R).
\]
\end{lemma}

\begin{proof}
The contraction argument constructs both the HNLS interaction representation
\eqref{eq:neg-int-rep} and the ODE flow \eqref{eq:ODE} on $[0,T_N]$ in the
space $\FL^1$, since $T_N\norm{\phi_N}_{\FL^1}^2=\beta_N$ is small.
Since the datum is a trigonometric polynomial, the same estimates in weighted
Wiener norms give persistence of smoothness on this interval.  Applying
Lemma~\ref{lem:wiener-bounds} to \eqref{eq:neg-int-rep} and to the ODE
Duhamel formula \eqref{eq:neg-int-rep-2} yields, uniformly for $0\le t\le T_N$,
\begin{align}
  \norm{\mathbf u_N(t)}_{\FL^1}
  +\norm{w_N(t)}_{\FL^1}
  &\lesssim \norm{\phi_N}_{\FL^1},\label{eq:neg-W1}\\
  \norm{\mathbf u_N(t)}_{\FL^\infty}
  +\norm{w_N(t)}_{\FL^\infty}
  &\lesssim \alpha_N,\\
  \norm{\mathbf u_N(t)-\phi_N}_{\FL^1}
  +\norm{w_N(t)-\phi_N}_{\FL^1}
  &\lesssim \beta_N\norm{\phi_N}_{\FL^1},\\
  \norm{\mathbf u_N(t)-\phi_N}_{\FL^\infty}
  +\norm{w_N(t)-\phi_N}_{\FL^\infty}
  &\lesssim \beta_N\alpha_N.\label{eq:neg-closeinf}
\end{align}

Let $A_{\mathrm{HNLS}}(t)$ and $A_{\mathrm{ODE}}(t)$ be the first Picard
corrections obtained by freezing the cubic term at $\phi_N$:
\[
  \widehat{A_{\mathrm{HNLS}}}(n,t)
  =
  i\sum_{n=n_1-n_2+n_3}
  \widehat{\phi_N}(n_1)\overline{\widehat{\phi_N}(n_2)}
  \widehat{\phi_N}(n_3)
  \int_0^t e^{it'\Phi(\bar n)}\,dt',
\]
and
\[
  \widehat{A_{\mathrm{ODE}}}(n,t)
  =
  it\sum_{n=n_1-n_2+n_3}
  \widehat{\phi_N}(n_1)\overline{\widehat{\phi_N}(n_2)}
  \widehat{\phi_N}(n_3).
\]
Using \eqref{eq:neg-W1}--\eqref{eq:neg-closeinf}, we first prove that
\begin{equation}\label{eq:neg-higher}
  \norm{\mathbf u_N(t)-\phi_N-A_{\mathrm{HNLS}}(t)}_{\FL^\infty}
  +
  \norm{w_N(t)-\phi_N-A_{\mathrm{ODE}}(t)}_{\FL^\infty}
  \lesssim \beta_N^2\alpha_N.
\end{equation}
We prove the first estimate in \eqref{eq:neg-higher}; the ODE estimate is the
same argument with the multiplier equal to $1$.  For each fixed $t'$, define
the trilinear multiplier
\[
  \mathcal F_x \big({\mathcal N_{t'}(f_1,f_2,f_3)} \big)(n)
  :=
  \sum_{n=n_1-n_2+n_3}
  e^{it'\Phi(\bar n)}
  \widehat f_1(n_1)\overline{\widehat f_2(n_2)}\widehat f_3(n_3).
\]
Since $\abs{e^{it'\Phi(\bar n)}}=1$, the bounds
\eqref{eq:tri-FL1}--\eqref{eq:tri-FLinf} apply to $\mathcal N_{t'}$.
Subtracting the frozen first Picard correction from the Duhamel formula gives
\[
\begin{aligned}
  \mathbf u_N(t)-\phi_N-A_{\mathrm{HNLS}}(t)
  &=
  i\int_0^t
  \Big[
  \mathcal N_{t'}\big(
    \mathbf u_N(t'),\mathbf u_N(t'),\mathbf u_N(t')\big)
  -\mathcal N_{t'}(\phi_N,\phi_N,\phi_N)
  \Big]\,dt'.
\end{aligned}
\]
By trilinearity,
\[
\begin{aligned}
&  \mathbf u_N(t')\overline{\mathbf u_N(t')}\mathbf u_N(t')
  -\phi_N\overline{\phi_N}\phi_N \\
&\quad =
  (\mathbf u_N(t')-\phi_N)\overline{\mathbf u_N(t')}\mathbf u_N(t')
  +\phi_N\overline{(\mathbf u_N(t')-\phi_N)}\,\mathbf u_N(t')
  +\phi_N\overline{\phi_N}(\mathbf u_N(t')-\phi_N).
\end{aligned}
\]
Using \eqref{eq:tri-FLinf}, together with \eqref{eq:neg-W1}--\eqref{eq:neg-closeinf} with $t' < T_N$
and $\norm{\phi_N}_{\FL^\infty}=\alpha_N$ from Lemma \ref{lem:neg-data-size},
we estimate the three terms respectively by
\[
\begin{aligned}
  \norm{\mathbf u_N(t')-\phi_N}_{\FL^\infty}
    \norm{\mathbf u_N(t')}_{\FL^1}^2
  &\lesssim \beta_N\alpha_N\norm{\phi_N}_{\FL^1}^2,\\
  \norm{\phi_N}_{\FL^\infty}
    \norm{\mathbf u_N(t')-\phi_N}_{\FL^1}
    \norm{\mathbf u_N(t')}_{\FL^1}
  &\lesssim \beta_N\alpha_N\norm{\phi_N}_{\FL^1}^2,\\
  \norm{\mathbf u_N(t')-\phi_N}_{\FL^\infty}
    \norm{\phi_N}_{\FL^1}^2
  &\lesssim \beta_N\alpha_N\norm{\phi_N}_{\FL^1}^2.
\end{aligned}
\]
Therefore, for $t \le T_N$, we have
\[
\begin{aligned}
&\norm{\mathbf u_N(t)-\phi_N-A_{\mathrm{HNLS}}(t)}_{\FL^\infty}\\
&\quad \le
  \int_0^t C\beta_N\alpha_N\norm{\phi_N}_{\FL^1}^2\,dt'
  \lesssim
  T_N\beta_N\alpha_N\norm{\phi_N}_{\FL^1}^2
  =
  \beta_N^2\alpha_N,
\end{aligned}
\]
since $T_N\norm{\phi_N}_{\FL^1}^2=\beta_N$.  For $w_N$, the same computation
applies with
$\mathbf u_N(t')$ replaced by $w_N(t')$ and with $e^{it'\Phi(\bar n)}$
replaced by $1$.
This proves \eqref{eq:neg-higher}.

It remains to compare the two frozen cubic corrections.  In each nonzero
summand, the input frequencies lie in $Q_N\cup Q_{2N}$, and hence the output
frequency also has size $\lesssim N$.  Since $q$ is quadratic,
\[
  \abs{\Phi(\bar n)}\lesssim N^2.
\]
Therefore, since $tN^2 \ll 1$, we have
\[
  \biggl|\int_0^t(e^{it'\Phi(\bar n)}-1)\,dt'\biggr|
  \lesssim t^2N^2.
\]
Using \eqref{eq:FLinf-alg} for the triple convolution of $\phi_N$, we obtain
\[
\begin{aligned}
  \norm{A_{\mathrm{HNLS}}(t)-A_{\mathrm{ODE}}(t)}_{\FL^\infty}
  \lesssim T_N^2N^2\alpha_N\norm{\phi_N}_{\FL^1}^2
  =
  (T_NN^2)\beta_N\alpha_N.
\end{aligned}
\]
Combining this with \eqref{eq:neg-higher},
\[
  \norm{\mathbf u_N-w_N}_{L^\infty([0,T_N];\FL^\infty)}
  \lesssim
  \beta_N^2\alpha_N+(T_NN^2)\beta_N\alpha_N
  =o(\beta_N\alpha_N).
\]
Finally,
\[
\begin{aligned}
  \norm{P_{Q_0}(\mathbf u_N(T_N)-w_N(T_N))}_{H^s}
  &\le
  \mathcal L_s(R)
  \norm{\mathbf u_N-w_N}_{L^\infty([0,T_N];\FL^\infty)}
  =
  o(\beta_N\alpha_N)\mathcal L_s(R).
\end{aligned}
\]
Thus, we finish the proof.
\end{proof}

\begin{proof}[Proof of Theorem~\ref{thm:main} when $s<0$]
We choose the parameters so that
\begin{align}\label{eq:condi}
  \norm{\phi_N}_{H^s(\T^2)}\to0,
  \qquad
  T_N\to0,
  \qquad
  T_NN^2\to0,
  \qquad
  \beta_N\alpha_N\mathcal L_s(R)\to\infty.
\end{align}
Let $R\sim N^\theta$ with $0<\theta<1$, set $\alpha_N R=N^\gamma$ so that $\alpha_N\sim N^{\gamma-\theta}$, and let
$\beta_N\to0$ be specified below.  By Lemma~\ref{lem:neg-data-size},
\[
  \norm{\phi_N}_{\FL^1}\sim \alpha_N R^2\sim N^{\gamma+\theta},
  \qquad
  \norm{\phi_N}_{H^s}\sim N^{\gamma+s}.
\]
Thus $\gamma+s<0$ makes the initial norm small, while
$\gamma+\theta>1$ makes
\[
  T_NN^2=\beta_N\norm{\phi_N}_{\FL^1}^{-2}N^2
  \sim \beta_N N^{2-2(\gamma+\theta)}
  \to0.
\]

If $-1<s<0$, choose
\[
  \theta\in(1+s,1),
  \qquad
  \gamma\in\big(\max\{1-\theta,-\theta s\},-s\big).
\]
Then $\gamma+s<0$, $\gamma+\theta>1$, and $\gamma+\theta s>0$.  Since
$\mathcal L_s(R)\sim R^{1+s}$ by \eqref{eq:neg-Ls},
\[
  \alpha_N\mathcal L_s(R)
  \sim N^{\gamma-\theta}N^{\theta(1+s)}
  =
  N^{\gamma+\theta s}\to\infty.
\]
Choose $\beta_N=N^{-\delta}$ with $0<\delta<\gamma+\theta s$.

If $s=-1$, choose $\theta=\gamma\in(\frac12,1)$.  Then
$\gamma+s = \gamma - 1<0$, $\gamma+\theta>1$, $\alpha_N\sim1$, and
$\mathcal L_{-1}(R)\sim(\log R)^{1/2}$ by \eqref{eq:neg-Ls}.  Taking
$\beta_N=(\log N)^{-1/4}$ gives
\[
  \beta_N\alpha_N\mathcal L_{-1}(R)\to\infty.
\]

If $s<-1$, choose any $\theta\in(0,1)$ and then choose
\[
  \gamma\in\big(\max\{\theta,1-\theta\},-s\big).
\]
This is possible because $-s>1$.  Then $\gamma+s<0$,
$\gamma+\theta>1$, and $\gamma-\theta>0$.  Since
$\mathcal L_s(R)\sim1$ by \eqref{eq:neg-Ls}, choose $\beta_N=N^{-\delta}$ with
$0<\delta<\gamma-\theta$.

In all three cases, the four desired properties \eqref{eq:condi} hold.
Lemmas
\ref{lem:neg-ODE-lower} and \ref{lem:neg-ODE-approx} give
\[
  \norm{P_{Q_0}\mathbf u_N(T_N)}_{H^s(\T^2)}
  \gtrsim_s \beta_N\alpha_N\mathcal L_s(R).
\]
Since the linear group $e^{it(\partial_x^2-\partial_y^2)}$ is unitary on
$H^s(\T^2)$,
\[
  \norm{u_N(T_N)}_{H^s(\T^2)}
  =
  \norm{\mathbf u_N(T_N)}_{H^s(\T^2)}
  \ge
  \norm{P_{Q_0}\mathbf u_N(T_N)}_{H^s(\T^2)}
  \to\infty.
\]
At the same time,
$\norm{u_N(0)}_{H^s(\T^2)}=\norm{\phi_N}_{H^s(\T^2)}\to0$
and $T_N\to0$.  Given $\varepsilon\in(0,1)$, taking $N$ sufficiently large
therefore gives
\[
  \norm{u_N(0)}_{H^s(\T^2)}<\varepsilon,
  \qquad
  0<T_N<\varepsilon,
  \qquad
  \norm{u_N(T_N)}_{H^s(\T^2)}>\varepsilon^{-1}.
\]
This proves Theorem~\ref{thm:main} for $s<0$.
\end{proof}

\section{The subcritical case \texorpdfstring{$0<s<\frac12$}{0<s<1/2}}
\label{sec:subcritical}
We now turn to the positive subcritical range $0<s<\frac12$.  Throughout this
section, fix $s\in(0,\frac12)$.  Let
\begin{align}
  D_N(z):=\sum_{\abs k\le N}e^{ikz}
  =\frac{\sin((N+\frac12)z)}{\sin(z/2)}
\end{align}
be the Dirichlet kernel and define
\begin{equation}\label{eq:phiN}
  \phi_N(z):=N^{-\frac12-s}D_N(z).
\end{equation}

\begin{lemma}[Size of the Dirichlet profile]\label{lem:phi-Hs}
Uniformly in $N$, we have $\norm{\phi_N}_{H^s(\T)}\sim_s1$.
\end{lemma}

\begin{proof}
Since $\widehat{D_N}(k)=1$ for $\abs k\le N$ and $0$ otherwise,
\[
  \norm{D_N}_{H^s(\T)}^2
  =\sum_{\abs k\le N}\la k\ra^{2s}
  \sim_s N^{1+2s}.
\]
Multiplying by $N^{-1-2s}$ gives the result.
\end{proof}

The next lemma isolates a fixed subinterval on the right-hand side of the
central peak of the Dirichlet kernel.  On this interval, which has length
comparable to $N^{-1}$, the profile $\phi_N$ has size $N^{\frac12-s}$ and its
derivative is negative with magnitude $N^{\frac32-s}$.  Consequently,
$-\partial_z\abs{\phi_N}^2$ has size $N^{2-2s}$.  The schematic in Figure \ref{fig:DN-interval} shows
the Dirichlet kernel $D_N(z)$, with the interval
$I_N$ marked near the central peak.

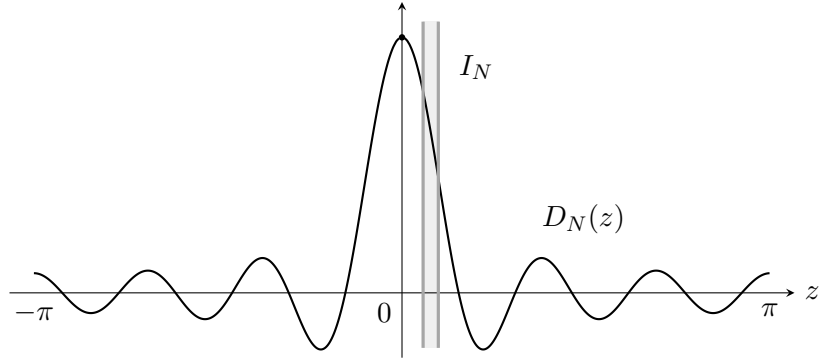
\begin{figure}[htbp]
\centering
\begin{tikzpicture}[x=1.55cm,y=0.26cm,>=stealth]
  \def\NN{6}
  \def\MM{13}
  \def\PI{3.14159265}
  \def\IA{0.18}
  \def\IB{0.31}
  \fill[gray!12] (\IA,-2.8) rectangle (\IB,13.8);
  \draw[->] (-3.35,0) -- (3.35,0) node[right] {$z$};
  \draw[->] (0,-3.3) -- (0,14.8);
  \draw[domain=-3.1416:-0.02,samples=360,smooth,variable=\z,thick]
    plot ({\z},{sin(180*\MM*\z/(2*\PI))/sin(180*\z/(2*\PI))});
  \draw[domain=0.02:3.1416,samples=360,smooth,variable=\z,thick]
    plot ({\z},{sin(180*\MM*\z/(2*\PI))/sin(180*\z/(2*\PI))});
  \fill (0,13) circle (1.2pt);
  \draw[gray!70,very thick] (\IA,-2.8) -- (\IA,13.8);
  \draw[gray!70,very thick] (\IB,-2.8) -- (\IB,13.8);
  \node[right] at (0.40,11.5) {$I_N$};
  \node[above] at (1.55,2.5) {$D_N(z)$};
  \node[below] at (-3.1416,0) {$-\pi$};
  \node[below left] at (0,0) {$0$};
  \node[below] at (3.1416,0) {$\pi$};
\end{tikzpicture}
\caption{The Dirichlet kernel $D_N$ over one full
period in the original variable $z$.  The two vertical lines mark the interval
$I_N=[11\pi/(12M),\pi/M]$ with $M=2N+1$ used in Lemma~\ref{lem:DN-interval}.}
\label{fig:DN-interval}
\end{figure}

\begin{lemma}[A Dirichlet interval]\label{lem:DN-interval}
There exist constants $c,C>0$ and $N_0\in\N$, depending only on $s$, such
that for every $N\ge N_0$ there is an interval
\[
  I_N\subset\Big[\frac cN,\frac CN\Big],
  \qquad
  \abs{I_N}\ge \frac cN,
\]
such that, for all $z\in I_N$,
\begin{align}
  cN^{\frac12-s}\le \phi_N(z)&\le CN^{\frac12-s},\label{eq:phi-size-all}\\
  cN^{\frac32-s}\le -\phi_N'(z)&\le CN^{\frac32-s},\label{eq:phi-prime-all}\\
  cN^{2-2s}\le -\partial_z\abs{\phi_N(z)}^2&\le CN^{2-2s}.
  \label{eq:phi-square-prime-all}
\end{align}
\end{lemma}

\begin{proof}
Set $M=2N+1$ and
\begin{align}\label{eq:IN}
  I_N:=\Big[\frac{11\pi}{12M},\frac{\pi}{M}\Big].
\end{align}
Then $\abs{I_N}\gtrsim N^{-1}$ and $I_N\subset[c/N,C/N]$.  For
$z\in I_N$, put $\theta=Mz/2$.  Then
\[
  \theta\in\Big[\frac{11\pi}{24},\frac{\pi}{2}\Big],
\]
so $\sin\theta$ is bounded above and below by positive constants, while
$\sin(z/2)\sim z\sim N^{-1}$.  Hence
\[
  D_N(z)=\frac{\sin(Mz/2)}{\sin(z/2)} =\frac{\sin(\theta)}{\sin(z/2)} \sim N.
\]
Differentiating,
\[
\begin{split}
  D_N'(z) & =
  \frac{\frac M2\cos(Mz/2)\sin(z/2)
  -\frac12\sin(Mz/2)\cos(z/2)}{\sin^2(z/2)} \\
  & =
  \frac{\frac M2\cos(\theta)\sin(z/2)
  -\frac12\sin(\theta)\cos(z/2)}{\sin^2(z/2)}.
 \end{split}
\]
On $I_N$, the numerator is negative and bounded away from $0$.  Indeed,
$M\sin(z/2)\le Mz/2=\theta\le\pi/2$, $\cos(z/2)\ge \sqrt3/2$, and
$\theta\in[11\pi/24,\pi/2]$.  Therefore
\[
\begin{aligned}
  &\frac M2\cos(\theta)\sin(z/2)
  -\frac12\sin(\theta)\cos(z/2) \\
  &\qquad\le
  \frac{\pi}{4}\cos\Big(\frac{11\pi}{24}\Big)
  -\frac{\sqrt3}{4}\sin\Big(\frac{11\pi}{24}\Big) \\
  &\qquad\le \frac{\pi^2}{96}-\frac38
  < -\frac14 .
\end{aligned}
\]
Here we used $\cos(11\pi/24)=\sin(\pi/24)\le \pi/24$ and
$\sin(11\pi/24)=\cos(\pi/24)\ge \sqrt3/2$.
Since $\sin(z/2)\sim N^{-1}$, this gives $-D_N'(z)\sim N^2$.  Because
$D_N$ is real on the real line,
\[
  \partial_z\abs{D_N(z)}^2=2D_N(z)D_N'(z)\sim -N^3.
\]
Multiplying by the powers in \eqref{eq:phiN} gives
\eqref{eq:phi-size-all}--\eqref{eq:phi-square-prime-all}.
\end{proof}

Lemma~\ref{lem:DN-interval} supplies the quantitative input for the following
lower bound, which is the main estimate in the subcritical argument.  We choose
a spatial increment $h$ so that the nonlinear phase
$t\rho^2\abs{\phi_N}^2$ changes by a definite amount on $I_N$, while the
amplitude $\phi_N$ changes only by a small relative amount.  The
fixed-increment estimate from Lemma~\ref{lem:fixed-increment} then converts
this pointwise separation into an $H^s$ lower bound.  The large parameter
driving the estimate is precisely $t\rho^2N^{1-2s}$.

\begin{proposition}[Subcritical lower bound]\label{prop:sub-lower}
There exist constants $c>0$, $\Lambda\ge1$, and $N_0\in\N$, depending only
on $s$, such that the following holds.  Let $N\ge N_0$, $0<\rho\le1$, and
$t>0$.  If
\[
  t\rho^2N^{1-2s}\ge \Lambda,
\]
then
\[
  \norm{\rho\phi_N e^{it\rho^2\abs{\phi_N}^2}}_{H^s(\T)}
  \ge c\rho\big(t\rho^2N^{1-2s}\big)^s.
\]
\end{proposition}

\begin{proof}
Let $I_N$ be the interval in Lemma~\ref{lem:DN-interval}.  Choose
$\alpha\in(0,1]$ so small that $C\alpha<2\pi$, where $C$ is the upper
constant in \eqref{eq:phi-square-prime-all}, and set
\begin{align}\label{eq:h}
  h:=\frac{\alpha}{t\rho^2N^{2-2s}}.
\end{align}
If $\Lambda$ is sufficiently large (depending only on $\alpha$), then
\[
  h\le \frac{\alpha}{\Lambda N}\le \frac{\abs{I_N}}4,
\]
where $I_N$ is the interval given in \eqref{eq:IN}.
Thus
\[
  J_N:=\{z\in I_N:\ z+h\in I_N\}
\]
satisfies $\abs{J_N}\gtrsim N^{-1}$.

For $z\in J_N$, the mean value theorem and
\eqref{eq:phi-square-prime-all} give
\[
  c\alpha
  \le
  t\rho^2\abs{\abs{\phi_N(z+h)}^2-\abs{\phi_N(z)}^2}
  \le C\alpha<2\pi,
\]
since $[z,z+h] \subset I_N$.
Hence \eqref{eq:phase-away} implies
\begin{align}\label{eq:low1}
  \abs{
  e^{it\rho^2\abs{\phi_N(z+h)}^2}
  -e^{it\rho^2\abs{\phi_N(z)}^2}
  }\ge c.
\end{align}
On the other hand, \eqref{eq:phi-prime-all} gives
\begin{align}\label{eq:up1}
  \abs{\phi_N(z+h)-\phi_N(z)}
  \le CN^{\frac32-s}h
  \le C\alpha (t\rho^2N^{1-2s})^{-1}N^{\frac12-s}.
\end{align}
Let
\[
  f(t,z):=\rho\phi_N(z)e^{it\rho^2\abs{\phi_N(z)}^2}.
\]
Taking $\Lambda$ large if necessary, we may assume that
\begin{align}\label{eq:com1}
  C\alpha(t\rho^2N^{1-2s})^{-1}N^{\frac12-s}
  \le C \alpha \Lambda^{-1} N^{\frac12-s}
  \ll N^{\frac12-s}.
\end{align}
Then, for every $z\in J_N$, the triangle inequality, together with
\eqref{eq:phi-size-all}, \eqref{eq:low1}, \eqref{eq:up1}, and
\eqref{eq:com1}, gives
\[
\begin{aligned}
  \abs{f(t,z+h)-f(t,z)}
  &=\rho\Big|
  \phi_N(z+h)e^{it\rho^2\abs{\phi_N(z+h)}^2}
  \\
  &\qquad\qquad
  -\phi_N(z)e^{it\rho^2\abs{\phi_N(z)}^2}
  \Big| \\
  &\ge \rho\phi_N(z)
  \Big|e^{it\rho^2\abs{\phi_N(z+h)}^2}
  \\
  &\qquad\qquad
  -e^{it\rho^2\abs{\phi_N(z)}^2}\Big|
  -\rho\abs{\phi_N(z+h)-\phi_N(z)} \\
  &\ge c\rho N^{\frac12-s}
  -C\alpha\rho
    \big(t\rho^2N^{1-2s}\big)^{-1}N^{\frac12-s} \\
  &\ge c\rho N^{\frac12-s}.
\end{aligned}
\]
Since $\abs{J_N}\gtrsim N^{-1}$ and the pointwise lower bound above holds
on $J_N$,
\[
  \norm{f(t,\cdot+h)-f(t,\cdot)}_{L^2(\T)}^2
  \ge \norm{f(t,\cdot+h)-f(t,\cdot)}_{L^2(J_N)}^2
  \ge c\rho^2N^{-2s}.
\]
Lemma~\ref{lem:fixed-increment} and \eqref{eq:h} give
\[
\begin{aligned}
  \norm{f(t)}_{\dot H^s(\T)}^2
  &\ge c h^{-2s} \norm{f(t,\cdot+h)-f(t,\cdot)}_{L^2(\T)}^2 \\
  &\ge c h^{-2s}\rho^2N^{-2s} \\
  &= c \alpha^{-2s} \big(t\rho^2N^{2-2s}\big)^{2s} \rho^2N^{-2s} \\
  &\ge c\rho^2\big(t\rho^2N^{2-2s}\big)^{2s}N^{-2s} \\
  &=c\rho^2\big(t\rho^2N^{1-2s}\big)^{2s}.
\end{aligned}
\]
Taking square roots proves the proposition.
\end{proof}

We can now prove Theorem~\ref{thm:main} when $0<s<\frac12$.

\begin{proof}[Proof of Theorem~\ref{thm:main} when $0<s<\frac12$]
Let $\varepsilon\in(0,1)$.  By Lemmas~\ref{lem:phi-Hs} and
\ref{lem:T2-to-T}, there is $C_s>0$ such that
\[
  \norm{\phi_N(x-y)}_{H^s(\T^2)}\le C_s
\]
for all $N$.  Choose $\rho=a_s\varepsilon$, where $a_sC_s<1$.  Then
\[
  \norm{\rho\phi_N(x-y)}_{H^s(\T^2)}<\varepsilon.
\]
Let $c_s>0$ be the constant obtained by combining
Proposition~\ref{prop:sub-lower} with Lemma~\ref{lem:T2-to-T}.  Choose
$A_s\ge\Lambda$ so large that
\begin{align}\label{eq:low2}
  c_s a_s A_s^s>1.
\end{align}
Finally choose $N$ so large that
\[
  t_\varepsilon:=
  \frac{A_s\varepsilon^{-2/s}}{\rho^2N^{1-2s}}<\varepsilon.
\]
This is possible since $1-2s>0$.  Set
\[
  u_0(x,y):=\rho\phi_N(x-y).
\]
Then $u_0\in C^\infty(\T^2)$, and the exact smooth solution is
\[
  u(t,x,y)=\rho\phi_N(x-y)
  e^{it\rho^2\abs{\phi_N(x-y)}^2}.
\]
Moreover
\[
  t_\varepsilon\rho^2N^{1-2s}
  =A_s\varepsilon^{-2/s}\ge\Lambda.
\]
Thus Proposition~\ref{prop:sub-lower} and \eqref{eq:low2} give
\[
\begin{aligned}
  \norm{u(t_\varepsilon)}_{H^s(\T^2)}
  &\ge c_s\rho
  \big(t_\varepsilon\rho^2N^{1-2s}\big)^s =c_s a_s A_s^s\varepsilon^{-1}
  >\varepsilon^{-1}.
\end{aligned}
\]
This proves the subcritical case.
\end{proof}

\section{The critical case \texorpdfstring{$s=\frac12$}{s=1/2}}
\label{sec:critical}

We finally treat the endpoint $s=\frac12$.  The subcritical Dirichlet-kernel
construction no longer gives a power gain, since the factor $N^{1-2s}$ becomes
$1$ at the endpoint.  We replace it by a logarithmic profile: after
normalization, the $\dot H^{1/2}$ norm stays bounded, while the pointwise size
on logarithmic spatial scales remains large enough to create a substantial
nonlinear phase.  The lower bound below captures this logarithmic accumulation
through the double-integral characterization of $\dot H^{1/2}(\T)$.

Let
\begin{equation} \label{eq:log}
  F_N(z):=\sum_{k=1}^N\frac{e^{ikz}}k,
  \qquad
  L:=\log N,
  \qquad
  g_N:=L^{-1/2}F_N.
\end{equation}
Then
\begin{equation}\label{eq:g-Hhalf-all}
  \norm{g_N}_{\dot H^{1/2}(\T)}^2
  =L^{-1}\sum_{k=1}^N\frac1k\sim1,
  \qquad
  \norm{g_N}_{H^{1/2}(\T)}\lesssim1.
\end{equation}
The logarithmic profile above \eqref{eq:log} is inspired by the endpoint counterexample of
Liu and Zheng~\cite{LiuZhengHNLS}.  Their result is stated for the cubic HNLS
on $\T^3$, but the relevant high-frequency profile is essentially
one-dimensional: after choosing the appropriate resonant direction, the core
object is a logarithmic Fourier sum of the form
$\sum_{k\le N}e^{ikz}/k$.  This is enough for their purpose, namely to show
that the solution map at the origin fails to be $C^3$ in the endpoint space.

For norm inflation, however, the third-order obstruction in \cite{LiuZhengHNLS} alone is not enough.
We need information on the full nonlinear phase
$e^{i\tau\abs{g_N}^2}$, not only on a Taylor coefficient.  The rest of this
section therefore studies the pointwise size and increment behavior of $F_N$
on logarithmic spatial scales.  In particular, the lower bound uses how
$\abs{F_N}^2$ changes between $z=e^{-m}$ and $e^{-(m+\eta)}$, and how these
changes accumulate in the double-integral formula for $\dot H^{1/2}(\T)$.
The schematic in Figure \ref{fig:FN-log-region} shows
the logarithmic profile $\operatorname{Re} F_N(z)$, with the interval
$I_N$ marked near the central peak.

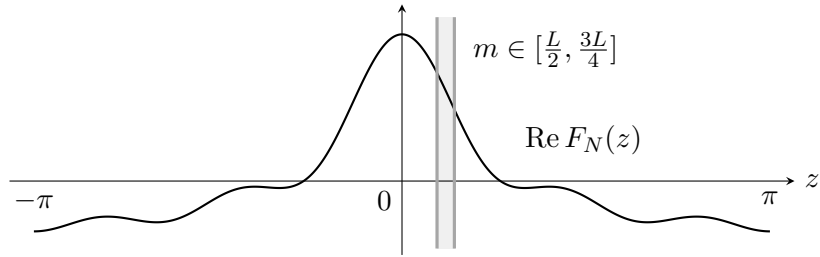
\begin{figure}[htbp]
\centering
\begin{tikzpicture}[x=1.55cm,y=0.85cm,>=stealth]
  \def\JA{0.299}
  \def\JB{0.447}
  \fill[gray!12] (\JA,-1.05) rectangle (\JB,2.55);
  \draw[->] (-3.35,0) -- (3.35,0) node[right] {$z$};
  \draw[->] (0,-1.15) -- (0,2.75);
  \draw[domain=-3.1416:3.1416,samples=500,smooth,variable=\z,thick]
    plot ({\z},{
      cos(180*\z/3.14159265)
      +cos(360*\z/3.14159265)/2
      +cos(540*\z/3.14159265)/3
      +cos(720*\z/3.14159265)/4
      +cos(900*\z/3.14159265)/5
    });
  \draw[gray!70,very thick] (\JA,-1.05) -- (\JA,2.55);
  \draw[gray!70,very thick] (\JB,-1.05) -- (\JB,2.55);
  \node[right] at (0.52,2.05) {$m\in[\frac L2,\frac{3L}{4}]$};
  \node[above] at (1.55,0.25) {$\operatorname{Re}F_N(z)$};
  \node[below] at (-3.1416,0) {$-\pi$};
  \node[below left] at (0,0) {$0$};
  \node[below] at (3.1416,0) {$\pi$};
\end{tikzpicture}
\caption{The real part of the logarithmic profile $F_N$ over one period.  The marked strip corresponds to the logarithmic region $z=e^{-m}$, $m\in[L/2,3L/4]$, used in
Lemma~\ref{lem:FN-asymptotics-all}.}
\label{fig:FN-log-region}
\end{figure}

\begin{lemma}[Logarithmic asymptotics]\label{lem:FN-asymptotics-all}
There exist absolute constants $c,C>0$ such that the following holds for all
sufficiently large $N$.  Let
\begin{align}
  m\in\Big[\frac L2,\frac{3L}{4}\Big],
  \qquad
  0<\eta\le1,
  \qquad
  z=e^{-m},
  \qquad
  z_\eta=e^{-(m+\eta)}.
\end{align}
Then
\begin{align}
  F_N(z) - m - i\frac\pi2 & = O(e^{-cL}),\\
  F_N(z_\eta)-F_N(z)&=\eta+O(e^{-cL}),\\
  cL\le\abs{F_N(z)}&\le CL,\\
  \bigl|
  \abs{F_N(z_\eta)}^2-\abs{F_N(z)}^2-(2m\eta+\eta^2)
  \bigr|&\le Ce^{-cL}L.
\end{align}
Consequently, for each fixed $\eta_*\in(0,1]$, there are constants
$c_{\eta_*},C_{\eta_*}>0$ such that
\[
  c_{\eta_*}L
  \le
  \abs{F_N(z_\eta)}^2-\abs{F_N(z)}^2
  \le C_{\eta_*}L
\]
whenever $\eta_*\le\eta\le1$.
\end{lemma}

\begin{proof}
We use the classical Fourier series, valid for $0<z<2\pi$,
\begin{equation}
  \sum_{k=1}^\infty \frac{\sin(kz)}k=\frac{\pi-z}{2},
  \qquad
  \sum_{k=1}^\infty \frac{\cos(kz)}k=-\log\big(2\sin(z/2)\big).
\end{equation}
Hence
\[
  \sum_{k=1}^\infty \frac{e^{ikz}}k
  =-\log\big(2\sin(z/2)\big)+i\Big(\frac\pi2-\frac z2\Big)
  =m+i\frac\pi2+O(e^{-L/2}),
\]
provided $z=e^{-m}$ and $m\ge L/2$.
It remains to estimate the tail of the series.
Let
\[
  S_M(z)=\sum_{k=1}^M e^{ikz}.
\]
By summation by parts, for $M>N$,
\[
\begin{aligned}
  \sum_{k=N+1}^M\frac{e^{ikz}}{k}
  &=
  \frac{S_M(z)}{M}
  -\frac{S_N(z)}{N+1}
  +\sum_{k=N+1}^{M-1}S_k(z)
    \biggl(\frac1k-\frac1{k+1}\biggr).
\end{aligned}
\]
Letting $M\to\infty$, since
$\abs{S_M(z)}\le C z^{-1}$ for $0<z\le1$, we obtain
\[
  \abs{\sum_{k>N}\frac{e^{ikz}}k}\le \frac{C}{Nz}=Ce^{-L+m},
\]
because $L=\log N$ and $z=e^{-m}$.
In the range $m\in[L/2,3L/4]$, this bound is $O(e^{-L/4})$.  The same
estimate holds at $z_\eta=e^{-(m+\eta)}$, since $\eta\le1$.
Therefore
\[
  F_N(z)=m+i\frac\pi2+O(e^{-cL}),
  \qquad
  F_N(z_\eta)=m+\eta+i\frac\pi2+O(e^{-cL}),
\]
This proves the first two estimates.  The size estimate follows from
$m\sim L$.  Finally,
\[
  \abs{F_N(z_\eta)}^2-\abs{F_N(z)}^2
  =(m+\eta)^2-m^2+O(e^{-cL}L)
  =2m\eta+\eta^2+O(e^{-cL}L),
\]
and the final consequence follows for $\eta_*\le\eta\le1$.
\end{proof}

\begin{proposition}[Endpoint lower bound]\label{prop:end-lower}
There exist constants $c>0$, $Q_0\ge1$, $\kappa>0$, and $N_0\in\N$ such
that the following holds.  Let $N\ge N_0$, $L=\log N$, and $\tau>0$.  If
\[
  Q:=\tau L,
  \qquad
  Q_0\le Q\le \kappa L,
\]
then
\[
  \norm{g_Ne^{i\tau\abs{g_N}^2}}_{\dot H^{1/2}(\T)}
  \ge cQ.
\]
\end{proposition}

\begin{proof}
Fix $\eta_0\in(0,1/10)$, and consider the set of pairs
\begin{align}\label{eq:range}
 R = \bigg \{(z,w)\in \T^2: z=e^{-m},
  \quad
  w=e^{-(m+\eta)},
  \quad
  m\in\Big[\frac L2,\frac{3L}{4}\Big],
  \quad
  \eta\in[\eta_0,2\eta_0] \bigg\}.
\end{align}
Here $0<w<z<1<\pi$, so $d_\T(z,w)=z-w$.
By Lemma~\ref{lem:FN-asymptotics-all}, there are constants $c_1,C_1>0$ such
that
\[
  \abs{g_N(z)}\ge c_1L^{1/2},
  \qquad
  \abs{g_N(w)-g_N(z)}\le C_1L^{-1/2}.
\]
Also, with
\[
  \Delta:=\tau\big(\abs{g_N(w)}^2-\abs{g_N(z)}^2\big),
\]
there are constants $c_2,C_2>0$ such that
\[
  c_2\frac{Q}{L}\le \Delta\le C_2\frac{Q}{L}.
\]
Choose $\kappa\le(2C_2)^{-1}$.  Then, whenever $Q\le\kappa L$,
\[
  0<\Delta\le\frac12,
\]
and hence
\[
  \abs{e^{i\Delta}-1}\ge c_3\Delta\ge c_2c_3\frac{Q}{L}
\]
for an absolute constant $c_3>0$.  Hence
\[
\begin{aligned}
&\abs{
g_N(w)e^{i\tau\abs{g_N(w)}^2}
-g_N(z)e^{i\tau\abs{g_N(z)}^2}
} \\
&\quad\ge
\abs{g_N(z)}\abs{e^{i\Delta}-1}
-\abs{g_N(w)-g_N(z)} \\
& \quad \ge (c_1c_2c_3Q-C_1)L^{-1/2}.
\end{aligned}
\]
Taking $Q_0$ sufficiently large, we obtain
\begin{equation}\label{eq:end-pointwise-all}
  \abs{
g_N(w)e^{i\tau\abs{g_N(w)}^2}
-g_N(z)e^{i\tau\abs{g_N(z)}^2}
}
  \ge c\frac{Q}{L^{1/2}}
\end{equation}
on the chosen region.

The change of variables $(m,\eta)\mapsto(z,w)$ satisfies
\[
  dz\,dw=e^{-2m-\eta}\,dm\,d\eta,
  \qquad
  (z-w)^2=e^{-2m}(1-e^{-\eta})^2.
\]
Since $\eta\in[\eta_0,2\eta_0]$,
\begin{align}\label{eq:cv}
  \frac{dz\,dw}{(z-w)^2}
  =\frac{e^{-\eta}}{(1-e^{-\eta})^2}\,dm\,d\eta
  \sim_{\eta_0}dm\,d\eta.
\end{align}
Although the physical $z$-interval is very small, $m$ is a logarithmic scale
variable, and the singular kernel $(z-w)^{-2}$ cancels the Jacobian loss on
pairs with $w=e^{-\eta}z$.

Let
\[
  H(z):=g_N(z)e^{i\tau\abs{g_N(z)}^2}.
\]
By Lemma~\ref{lem:Hs-difference} with $s=\frac12$,
\[
  \norm{H}_{\dot H^{1/2}(\T)}^2
  \sim
  \iint_{\T\times\T}
  \frac{\abs{H(z)-H(w)}^2}{d_\T(z,w)^2}\,dz\,dw.
\]
We restrict this positive double integral to the set $R$ in \eqref{eq:range}.
Using $d_\T(z,w)=\abs{z-w}$ on $R$, the change of variables \eqref{eq:cv}, and
the pointwise lower bound \eqref{eq:end-pointwise-all}, we get
\[
\begin{aligned}
  \norm{H}_{\dot H^{1/2}(\T)}^2
  &\gtrsim
  \iint_{R}
  \frac{\abs{H(z)-H(w)}^2}{|z-w|^2}\,dz\,dw \\
  & =
  \int_{L/2}^{3L/4}\int_{\eta_0}^{2\eta_0}
  \abs{H(e^{-(m+\eta)})-H(e^{-m})}^2
  \frac{e^{-\eta}}{(1-e^{-\eta})^2}\,d\eta\,dm \\
  &\gtrsim_{\eta_0}
  \int_{L/2}^{3L/4}\int_{\eta_0}^{2\eta_0}
  \frac{Q^2}{L}\,d\eta\,dm
  \gtrsim_{\eta_0}
  \frac{Q^2}{L}\cdot L
  \gtrsim Q^2.
\end{aligned}
\]
Taking square roots proves the proposition.
\end{proof}

\begin{proof}[Proof of Theorem~\ref{thm:main} when $s=\frac12$]
Let $\varepsilon\in(0,1)$.  By \eqref{eq:g-Hhalf-all} and
Lemma~\ref{lem:T2-to-T}, there is $C_*>0$ such that
\[
  \norm{g_N(x-y)}_{H^{1/2}(\T^2)}\le C_*
\]
for all sufficiently large $N$.  Choose $\rho=a\varepsilon$ with $aC_*<1$.
Then
\[
  \norm{\rho g_N(x-y)}_{H^{1/2}(\T^2)}<\varepsilon.
\]
Let $c_{\mathrm{end}}>0$ be the constant obtained by combining
Proposition~\ref{prop:end-lower}, Lemma~\ref{lem:T2-to-T}, and the trivial
lower bound $\norm{h}_{H^{1/2}}\ge\norm{h}_{\dot H^{1/2}}$.  Choose
\[
  Q:=A\varepsilon^{-2},
\]
where $A\ge Q_0$ is so large that
\[
  c_{\mathrm{end}}aA>1.
\]
Choose $N\ge N_0$ so large that
\[
  L=\log N
  \ge
  \max\biggl\{\frac{2Q}{\kappa},\frac{2Q}{\rho^2\varepsilon}\biggr\},
\]
and define
\[
  t_\varepsilon:=\frac{Q}{\rho^2L}.
\]
Here $Q$ is the large parameter in Proposition~\ref{prop:end-lower}, whereas
the phase parameter is $\tau=t_\varepsilon\rho^2=Q/L$.  Increasing $N$ allows
us to keep $\tau$ small even though $Q$ is large.
Then $0<t_\varepsilon<\varepsilon$ and
\[
  t_\varepsilon\rho^2L=Q\le \kappa L.
\]
Set
\[
  u_0(x,y):=\rho g_N(x-y).
\]
The exact solution is
\[
  u(t,x,y)=\rho g_N(x-y)
  e^{it\rho^2\abs{g_N(x-y)}^2}.
\]
Applying Proposition~\ref{prop:end-lower} with
$\tau=t_\varepsilon\rho^2$ gives
\[
  \norm{u(t_\varepsilon)}_{H^{1/2}(\T^2)}
  \ge c_{\mathrm{end}}\rho Q
  =c_{\mathrm{end}}aA\varepsilon^{-1}
  >\varepsilon^{-1}.
\]
This proves the endpoint case. 

\smallskip

Together with the negative and subcritical
constructions in sections \ref{sec:negative} and \ref{sec:subcritical}, this proves norm inflation for
$s\in(-\infty,0)\cup(0,\frac12]$. Conversely, when $s=0$, norm inflation is
excluded by conservation of the $L^2$ norm. When $s>\frac12$, the analytic
local well-posedness result of \cite{WangHNLS} gives a solution map continuous
at the origin, and
therefore precludes norm inflation at the origin. This completes the proof of
Theorem~\ref{thm:main}.
\end{proof}

\begin{ackno}\normalfont
S.S. was supported in part by the National Key R\&D Program of China under Grant 2024YFA1015500, NSF of China under Grant 12501322, and Anhui Provincial NSF 2508085QA001. Y.W. was supported by the EPSRC Mathematical Sciences Small
Grant (grant no. UKRI1116).
\end{ackno}

\end{document}